\documentclass[12pt,a4paper,twoside]{article}
\usepackage[]{graphicx}
\usepackage{pstricks,pst-text,pst-node,pst-plot}
\usepackage{epsf}
\usepackage{amssymb}
\usepackage{amsmath}
\newtheorem{theorem}{Theorem}
\newtheorem{corollary}[theorem]{Corollary}

\newtheorem{example}[theorem]{Example}
\newtheorem{lemma}[theorem]{Lemma}
\newtheorem{proposition}[theorem]{Proposition}

\newenvironment{proof}[1][Proof]{\textbf{#1.} }{\
\rule{0.5em}{0.5em}}

\title{Solutions of nonlinear PDEs in the completion of uniform convergence spaces}

\author{Jan Harm van der Walt\\Department of Mathematics and Applied
Mathematics\\University of Pretoria, South Africa}

\date{}
\begin{document}

\maketitle

\begin{abstract}
This paper deals with the solution of large classes of systems of
nonlinear partial differential equations (PDEs) in spaces of
generalized functions that are constructed as the completion of
uniform convergence spaces. The existence result for the mentioned
systems of equations are obtained as an application of a basic
approximation result, which is formulated entirely in terms of
usual real valued functions on open subsets of Euclidean
$n$-space.  The structure and regularity properties of the
solutions are explained at the hand of suitable results relating
to the structure of the completion of uniform convergence spaces
that are defined as initial structures.  In this regard, we
include also a detailed discussion of the completion of initial
uniform convergence spaces in general.
\end{abstract}

\section{Introduction}

Uniform spaces, and more generally uniform convergence spaces,
appear in many important applications of topology, and in
particular analysis. In this regard, the concepts of completeness
and completion of a uniform convergence space play a central role.
Indeed, Baire's celebrated Category Theorem asserts that a
\textit{complete} metric space cannot be expressed as the union of
a countable family of closed nowhere dense sets. The importance of
this result is demonstrated by the fact that the Banach-Steinhauss
Theorem, as well as the Closed Graph Theorem in Banach spaces
follow from it.

However, in many situations one deals with a space $X$ which is
\textit{incomplete}, and in these cases one may want to construct
the \textit{completion} of $X$.  In this regard, the main result,
see for instance \cite{Gahler 2} and \cite{Wyler}, is that every
Hausdorff uniform convergence space $X$ may be uniformly
continuously embedded into a \textit{complete}, Hausdorff uniform
convergence space $X^{\sharp}$ in such a way that the image of $X$
in $X^{\sharp}$ is dense. Moreover, the following
\textit{universal property} is satisfied. For every complete,
Hausdorff uniform convergence space $Y$, and any uniformly
continuous mapping
\begin{eqnarray}
\varphi:X\rightarrow Y\nonumber
\end{eqnarray}
the diagram
\begin{eqnarray}
\begin{array}{c}
\setlength{\unitlength}{1cm} \thicklines
%\begin{pspicture}(13,6)
\begin{picture}(13,6)

\put(1.4,4.9){$X$} \put(2.0,5.0){\vector(1,0){7.4}}
\put(9.6,4.9){$Y$} \put(5.8,5.4){$\varphi$}
\put(1.7,4.6){\vector(1,-1){3.5}} \put(6.0,1.1){\vector(1,1){3.5}}
\put(5.4,0.6){$X^{\sharp}$} \put(3.1,2.5){$\iota_{X}$}
\put(8.0,2.5){$\exists !\varphi^{\sharp}$}

\end{picture}
%\end{pspicture}
\end{array}\label{Univeral}
\end{eqnarray}
commutes, with $\varphi^{\sharp}$ uniformly continuous, and
$\iota_{X}$ the canonical embedding of $X$ into its completion
$X^{\sharp}$.

It is often not only the completion $X^{\sharp}$ of a uniform
convergence space $X$ that is of interest, but also the extension
$\varphi^{\sharp}$ of uniformly continuous mappings from $X$ to
$X^{\sharp}$. In this regard, we may recall that one of the major
applications of uniform spaces, and recently, see \cite{vdWalt2},
\cite{vdWalt6} and \cite{vdWalt3}, also uniform convergence
spaces, is to the solutions of PDEs. Indeed, let us consider a PDE
\begin{eqnarray}
Tu=f\label{PDE}
\end{eqnarray}
with $T$ a possibly nonlinear partial differential operator which
acts on some relatively small space $X$ of classical functions,
$u$ the unknown function, while the right hand term $f$ belongs to
some space $Y$. One usually considers some uniformities, or more
generally uniform convergence structures, on $X$ and $Y$ in such a
way that the mapping
\begin{eqnarray}
T:X\rightarrow Y\label{PDEOp}
\end{eqnarray}
is uniformly continuous.  It is well known that the equation
(\ref{PDE}) can have solutions of \textit{physical interest}
which, however, may fail to be \textit{classical}, in the sense
that they do not belong to $X$. From here, therefore, the
particular interest in \textit{generalized solutions} to
(\ref{PDE}).  Such generalized solutions to (\ref{PDE}) may be
obtained by constructing the completions $X^{\sharp}$ and
$Y^{\sharp}$ of $X$ and $Y$, respectively.  The mapping
(\ref{PDEOp}) extends uniquely to a mapping
\begin{eqnarray}
T^{\sharp}:X^{\sharp}\rightarrow Y^{\sharp} \label{ExtPDEOp}
\end{eqnarray}
so that the diagram
\begin{eqnarray}
\begin{array}{c}
\setlength{\unitlength}{1cm} \thicklines
%\begin{pspicture}(13,6)
\begin{picture}(13,6)

\put(1.9,4.9){$X$} \put(2.3,5.0){\vector(1,0){6.0}}
\put(8.5,4.9){$Y$} \put(4.9,5.2){$T$}
\put(2.0,4.7){\vector(0,-1){3.5}} \put(2.4,0.9){\vector(1,0){6.0}}
\put(1.8,0.8){$X^{\sharp}$} \put(8.5,0.8){$Y^{\sharp}$}
\put(1.4,2.9){$\iota_{X}$} \put(8.8,2.9){$\iota_{Y}$}
\put(8.6,4.7){\vector(0,-1){3.5}} \put(4.9,1.1){$T^{\sharp}$}

\end{picture}
%\end{pspicture}
\end{array}\label{PDEDiagram}
\end{eqnarray}
commutes.  In view of the diagram (\ref{PDEDiagram}), one may
consider the \textit{extended} equation
\begin{eqnarray}
T^{\sharp}u^{\sharp}=f\label{ExtPDE}
\end{eqnarray}
as a generalization of the original equation (\ref{PDE}).  That
is, the generalized solutions of (\ref{PDE}) are the solutions of
(\ref{ExtPDE}). Note that the \textit{existence} of generalized
solutions depends on the properties of the mapping $T^{\sharp}$
and the uniform convergence structure on $X^{\sharp}$ and
$Y^{\sharp}$, as opposed to the \textit{structure} and
\textit{regularity} of the generalized solutions, which may be
interpreted as the extent to which a generalized solution exhibits
characteristics of classical solutions, which depends on the
properties of the space $X^{\sharp}$ and its elements.  It is
therefore clear that not only the \textit{completion} $X^{\sharp}$
of a uniform convergence space $X$ and its elements, but also the
the associated \textit{extensions} of uniformly continuous
mappings, defined on $X$, are of interest.

The example given above indicates a particular point of interest.
The uniform convergence structure $\mathcal{J}_{X}$ on the domain
$X$ of the PDE operator $T$ is usually defined as the
\textit{initial} uniform convergence structure \cite{Beattie} with
respect to some uniform convergence structure $\mathcal{J}_{Y}$ on
$Y$, and a family of mappings
\begin{eqnarray}
\left(\psi_{i}:X\rightarrow Y\right)_{i\in I}\nonumber
\end{eqnarray}
In the case of PDEs, the mappings $\psi_{i}$ are typically usual
partial differential operators, up to a given order $m$.  A
natural question arises as to the connection between the
completion of $X$, and the completion of $Y$.  More generally,
consider a set $X$ and a family of mappings
\begin{eqnarray}
\left(\psi_{i}:X\rightarrow X_{i}\right)_{i\in I}\nonumber
\end{eqnarray}
where each $X_{i}$ is a uniform convergence space.  If the family
$\left(\psi_{i}\right)_{i\in I}$ separates the points of $X$, then
the initial uniform convergence structure on $X$ with respect to
this family of mappings is also Hausdorff, and we may consider its
completion $X^{\sharp}$.  It appears that the issue of the
possible connections between the completions of $X$ and those of
the $X_{i}$, respectively, has not yet been fully explored. We aim
to clarify the connection between the completion $X^{\sharp}$ of
$X$, and the completions $X_{i}^{\sharp}$ of the $X_{i}$.

Regarding the above example concerning possibly nonlinear PDEs, we
note that the uniform structure on the target space $Y$ is usually
induced by some locally convex linear space topology on $Y$, while
the initial uniform structure on $X$ is defined in terms of usual
\textit{linear} partial differential operator.  Indeed, the
Sobolev space $H^{1}\left(\Omega\right)$, for instance, may be
defined as the \textit{completion} of the initial uniform
structure on $\mathcal{C}^{1}\left(\Omega\right)$ with respect to
the family of mappings
\begin{eqnarray}
\left(D^{\alpha}:\mathcal{C}^{1}\rightarrow
\mathcal{L}_{2}\left(\Omega\right)\right)_{|\alpha|\leq
1}\nonumber
\end{eqnarray}
These methods, however, fail to deliver the existence of
generalized solutions to any significantly general class of PDEs,
particularly in the nonlinear case.  This is not due to any
conceptual obstacles, and even less so to the limitations of
mathematics as such, but rather to the inherent limitations of the
linear function analytic methods themselves.

Indeed, the general and type independent theory
\cite{Oberguggenberger} for the existence of solutions to
nonlinear PDEs delivers generalized solutions to very large
classes of equations as elements of the Dedekind completion of
suitably constructed partially ordered sets.  What is more, one
obtains a \textit{blanket regularity} for these generalized
solutions, as they may be assimilated with Hausdorff continuous,
interval valued functions \cite{Anguelov and Rosinger 1}.  As an
application of the results on the completion of uniform
convergence spaces, we present a significant enrichment of the
basic theory of Order Completion \cite{Oberguggenberger}. In this
regard, we obtain the \textit{existence} and \textit{uniqueness}
of generalized solutions of $\mathcal{C}^{k}$-smooth systems of
nonlinear PDEs, which may be assimilated with functions which are
$\mathcal{C}^{k}$-smooth everywhere except on some closed nowhere
dense set.

The paper is organized as follows.  In Section 2 we discuss the
structure of the completion $Y^{\sharp}$ of a subspace $Y$ of a
uniform convergence structure $X$ relative to that of the
completion $X^{\sharp}$ of $X$. Section 3 addresses the structure
of the completion of the Cartesian product of a family of uniform
convergence spaces. In particular, we show that the Wyler
completion preserves Cartesian products. As an application of the
results on subspaces and products of uniform convergence spaces,
we investigate the structure of the completion of the initial
uniform convergence structure on a set $X$, with respect to
uniform convergence structures $\mathcal{J}_{X_{i}}$ on sets
$X_{i}$, and a family of mappings
\begin{eqnarray}
\left(\psi_{i}:X\rightarrow X_{i}\right)_{i\in I}\nonumber
\end{eqnarray}
in Section 4.  In the context of nonlinear PDEs, as explained
above, the results we obtain in this regard may be considered as a
regularity result. In Section 5 we apply the results of the
preceding sections to systems of nonlinear PDEs.

\section{Subspaces and Embeddings}

It can easily be shown that the Bourbaki completion of a uniform
space $X$ preserves subspaces.  In particular, the completion
$Y^{\sharp}$ of any subspace of $X$ is isomorphic to a subspace of
the completion $X^{\sharp}$ of $X$.  For uniform convergence
spaces in general, and the associated Wyler completion
\cite{Wyler}, this is not the case. In this regard, consider the
following\footnote{This example was communicated to the author by
Prof. H. P. Butzmann}.
\begin{example}\label{QRExample}
Consider the real line $\mathbb{R}$ equipped with the uniform
convergence structure associated with the usual uniformity.  Also
consider the set $\mathbb{Q}$ of rational numbers equipped with
the subspace uniform convergence structure induced from
$\mathbb{R}$. The Wyler completion $\mathbb{Q}^{\sharp}$ of
$\mathbb{Q}$ is the set $\mathbb{R}$ equipped with a suitable
uniform convergence structure.  As such, the inclusion mapping
$i:\mathbb{Q} \rightarrow \mathbb{R}$ extends to a uniformly
continuous bijection
\begin{eqnarray}
i^{\sharp}:\mathbb{Q}^{\sharp} \rightarrow
\mathbb{R}\label{QSharpMap}
\end{eqnarray}
Furthermore, a filter $\mathcal{F}$ on $\mathbb{Q} ^{\sharp}$
converges to $x^{\sharp}$ if and only if
\begin{eqnarray}
\mathcal{V}\left(x^{\sharp}\right)_{|\mathbb{Q}} \cap
[x^{\sharp}]\subseteq \mathcal{F}
\end{eqnarray}
where $\mathcal{V}\left(x^{\sharp}\right)$ denotes the
neighborhood filter at $x^{\sharp}$ in $\mathbb{R}$.  As such, it
is clear that the neighborhood filter at $x^{\sharp}$ does not
converge in $\mathbb{Q}^{\sharp}$.  Therefore the mapping
(\ref{QSharpMap}) does not have a continuous inverse, so that it
is not an embedding.
\end{example}

In view of Example \ref{QRExample}, it is clear that Wyler
completion does not preserve subspaces.  Indeed, even in case the
uniform convergence structure is induced by a uniformity, the
completion of a subspace of a uniform convergence space $X$ will
in general not be a subspace of the completion $X^{\sharp}$.
Before we proceed to establish with our investigation of the
completion of a subspace $Y$ of a uniform convergence structure
$X$, we note that the Wyler completion is minimal among complete
uniform convergence spaces containing a given uniform convergence
space $X$ as a dense subspace.  Indeed, this is an easy
consequence of the universal property (\ref{Univeral}).
\begin{proposition}\label{WylerMinimal}
Consider a Hausdorff uniform convergence space $X$.  For any
complete, Hausdorff uniform convergence space $X^{\sharp}_{0}$
that contains $X$ a dense subspace, there is a bijective and
uniformly continuous mapping
\begin{eqnarray}
\iota_{X,0}^{\sharp}:X^{\sharp}\rightarrow X^{\sharp}_{0}.
\end{eqnarray}
\end{proposition}
\begin{proof}
Consider the inclusion mapping
\begin{eqnarray}
i_{X,0}:X\rightarrow X^{\sharp}_{0},\label{Incl}
\end{eqnarray}
which is clearly a uniformly continuous embedding.  As such, there
is a unique uniformly continuous mapping
\begin{eqnarray}
\iota_{X,0}^{\sharp}:X^{\sharp}\rightarrow
X^{\sharp}_{0}\label{ExtIncl}
\end{eqnarray}
which extends the mapping (\ref{Incl}).  We show that
(\ref{ExtIncl}) is a bijection.  In this regard, consider any
point $x^{\sharp}_{0}\in X^{\sharp}_{0}$.  Since $X$ is dense in
$X^{\sharp}_{0}$, there is some Cauchy filter $\mathcal{F}$ on $X$
such that $[\mathcal{F}]_{X^{\sharp}_{0}}$ converges to
$x^{\sharp}_{0}$ in $X^{\sharp}_{0}$.  Furthermore, there is some
$x^{\sharp}\in X^{\sharp}$ so that $[\mathcal{F}]_{X^{\sharp}}$
converges to $x^{\sharp}$ in $X^{\sharp}$.  Since
\begin{eqnarray}
i_{X,0}^{\sharp}\left([\mathcal{F}]_{X^{\sharp}}\right)=
i_{X,0}\left(\mathcal{F}\right) =
[\mathcal{F}]_{X^{\sharp}_{0}}\label{FEq}
\end{eqnarray}
it follows that $i_{X,0}^{\sharp}\left(x^{\sharp}\right) =
x^{\sharp}_{0}$.  Hence (\ref{ExtIncl}) is surjective.\\
To see that (\ref{ExtIncl}) is injective, consider  any
$x^{\sharp},y^{\sharp}\in X^{\sharp}$ and Cauchy filters
$\mathcal{F}$ and $\mathcal{G}$ on $X$ which converge to
$x^{\sharp}$ and $y^{\sharp}$, respectively, in $X^{\sharp}$.
Suppose that
\begin{eqnarray}
i_{X,0}^{\sharp}\left(x^{\sharp}\right)=i_{X,0}^{\sharp}\left(y^{\sharp}\right)
= z_{0}^{\sharp}
\end{eqnarray}
for some $z_{0}^{\sharp}\in X_{0}^{\sharp}$.  It now follows by
(\ref{FEq}) that $[\mathcal{F}]_{X^{\sharp}_{0}}$ and
$[\mathcal{G}]_{X^{\sharp}_{0}}$ both converge to $z_{0}^{\sharp}$
in $X_{0}^{\sharp}$.  As such, $[\mathcal{F}]_{X^{\sharp}_{0}}\cap
[\mathcal{G}]_{X^{\sharp}_{0}}$ converges to $z_{0}^{\sharp}$ so
that $\mathcal{F}\cap \mathcal{G}$ is a Cauchy filter on $X$. This
shows that $x_{0}^{\sharp}=y_{0}^{\sharp}$ so that (\ref{ExtIncl})
is injective.  This completes the proof.
\end{proof}\\ \\
The main result of this section is the following.
\begin{proposition}\label{SubspUCSComp}
Let $Y$ be a subspace of the uniform convergence space $X$.  Then
there is an injective, uniformly continuous mapping
\begin{eqnarray}
i^{\sharp}:Y^{\sharp} \rightarrow X^{\sharp}
\end{eqnarray}
which extends the inclusion mapping $i:Y\rightarrow X$.  In
particular,
\begin{eqnarray}
i^{\sharp}\left(Y^{\sharp}\right) =
a_{X^{\sharp}}\left(\iota_{X}\left(Y\right)\right).
\end{eqnarray}
Furthermore, the uniform convergence structure on $Y^{\sharp}$ is
the smallest complete, Hausdorff uniform convergence structure on
$a_{X^{\sharp}}\left(Y\right)$, with respect to inclusion, so that
$Y$ is contained as a dense subspace.
\end{proposition}
\begin{proof}
In view of the fact that the inclusion mapping $i:Y\rightarrow X$
is a uniformly continuous embedding, we obtain a uniformly
continuous mapping
\begin{eqnarray}
i^{\sharp}:Y^{\sharp}\rightarrow X^{\sharp}\label{IncExt}
\end{eqnarray}
so that the diagram
\begin{eqnarray}
\begin{array}{c}
\setlength{\unitlength}{1cm} \thicklines
%\begin{pspicture}(13,6)
\begin{picture}(13,6)

\put(2.9,4.9){$Y$} \put(3.5,5.0){\vector(1,0){6.4}}
\put(10.3,4.9){$X$} \put(6.3,5.2){$i$}
\put(3.0,4.7){\vector(0,-1){3.5}} \put(3.5,0.7){\vector(1,0){6.4}}
\put(2.9,0.6){$Y^{\sharp}$} \put(10.1,0.6){$X^{\sharp}$}
\put(2.3,2.9){$\iota_{Y}$} \put(10.5,2.9){$\iota_{X}$}
\put(10.2,4.7){\vector(0,-1){3.5}} \put(6.3,0.9){$i^{\sharp}$}

\end{picture}
%\end{pspicture}
\end{array}\label{IncDiagram}
\end{eqnarray}
commutes.  To see that the mapping (\ref{IncExt}) is injective,
consider any $y^{\sharp}_{0},y^{\sharp}_{1}\in Y^{\sharp}$ and
suppose that
\begin{eqnarray}
i^{\sharp}\left(y^{\sharp}_{0}\right)=i^{\sharp}\left(y^{\sharp}_{1}\right)=x^{\sharp}
\label{ConAss}
\end{eqnarray}
for some $x^{\sharp}\in X^{\sharp}$.  Since
$\iota_{Y}\left(Y\right)$ is dense in $Y^{\sharp}$ there exists
Cauchy filters $\mathcal{F}$ and $\mathcal{G}$ on $Y$ such that
$\iota_{Y}\left(\mathcal{F}\right)$ converges to $y^{\sharp}_{0}$
and $\iota_{Y}\left(\mathcal{G}\right)$ converges to
$y^{\sharp}_{1}$. From the diagram above it follows that
$\iota_{X}\left(i\left(\mathcal{F}\right)\right)$ and
$\iota_{X}\left(i\left(\mathcal{G}\right)\right)$ converges to
$x^{\sharp}$.  Therefore the filter
\begin{eqnarray}
\mathcal{H}=\iota_{X}\left(i\left(\mathcal{F}\right)\right)\cap
\iota_{X}\left(i\left(\mathcal{G}\right)\right)\nonumber
\end{eqnarray}
converges to $x^{\sharp}$ in $X^{\sharp}$.  Note that the filter
\begin{eqnarray}
i^{-1}\left(\iota_{X}^{-1}\left(\mathcal{H}\right)\right)\nonumber
\end{eqnarray}
is a Cauchy filter on $Y$ so that $\iota_{Y}\left(
i^{-1}\left(\iota_{X}^{-1}\left(\mathcal{H}\right)\right)\right)$
must converge in $Y^{\sharp}$ to some $y^{\sharp}$.  But
$\iota_{Y}\left( i^{-1}\left(\iota_{X}^{-1}
\left(\mathcal{H}\right)\right)\right)\subseteq
\iota_{Y}\left(\mathcal{F}\right)$ and $\iota_{Y}\left(
i^{-1}\left(\iota_{X}^{-1}\left(\mathcal{H}\right)\right)\right)\subseteq
\iota_{Y}\left(\mathcal{G}\right)$ so that
$\iota_{Y}\left(\mathcal{F}\right)$ and
$\iota_{Y}\left(\mathcal{G}\right)$ must converge to $y^{\sharp}$
as well.  Since $Y^{\sharp}$ is Hausdorff it follows by
(\ref{ConAss}) that $y^{\sharp}_{0}=y^{\sharp}_{1}=y^{\sharp}$.
Therefore $i^{\sharp}$ is injective.\\
Clearly $i^{\sharp}\left(Y^{\sharp}\right) \subseteq
a_{X^{\sharp}}\left(\iota_{X}\left(Y\right)\right)$.  To verify
the reverse inclusion, consider any $x^{\sharp} \in a_{X^{\sharp}}
\left( \iota_{X}\left(Y\right)\right)$.  Then
\begin{eqnarray}
\begin{array}{ll}
\exists & \mathcal{F}\mbox{ a filter on $\iota_{X}\left(Y\right)$ :} \\
& [\mathcal{F}]_{X^{\sharp}}\mbox{ converges to $x^{\sharp}$ in $X^{\sharp}$} \\
\end{array}.
\end{eqnarray}
Then there is a Cauchy filter $\mathcal{G}$ on $X$ so that
\begin{eqnarray}
\iota_{X}\left(\mathcal{G}\right)\cap[x^{\sharp}]\subseteq
[\mathcal{F}]_{X^{\sharp}}
\end{eqnarray}
This implies that the Cauchy filter $\mathcal{G}$ has a trace
$\mathcal{H}=\mathcal{G}_{|Y}$ on $Y$, which is a Cauchy filter on
$Y$.  The result now follows by the commutative diagram
(\ref{IncDiagram}).\\
The last statement follows immediately from Proposition
\ref{WylerMinimal}.
\end{proof}\\ \\
The following is an immediate consequence of Proposition
\ref{SubspUCSComp}.
\begin{corollary}\label{EmbExt}
Let $X$ and $Y$ be uniform convergence spaces, and
$\varphi:X\rightarrow Y$ a uniformly continuous embedding.  Then
there exists an injective uniformly continuous mapping
$\varphi^{\sharp}:X^{\sharp}\rightarrow Y^{\sharp}$, where
$X^{\sharp}$ and $Y^{\sharp}$ are the completions of $X$ and $Y$
respectively, which extends $\varphi$.
\end{corollary}

It should be noted that there are many different `completions'
which one may associate with a given Hausdorff uniform convergence
space, each designed so as to preserve a specific property, or
properties, of a uniform convergence space.  E. Reed \cite{Reed}
made a definitive study of several such completions. Furthermore,
the completion of a convergence vector space \cite{GGK}, the
completion of a convergence group \cite{Fric and Kent}, and the
Wyler completion \cite{Wyler} of a uniform convergence space are
in general all different.  Indeed, the Wyler completion is
typically not compatible with the algebraic structure of a
convergence group or convergence vector space \cite{Beattie},
while the convergence group completion of a convergence vector
space does in general not induce a vector space convergence
structure \cite{Beattie}.  Among all possible completions, the
Wyler completion is the only one that satisfies the universal
property (\ref{Univeral}).  None of the mentioned completions,
however, will, in general, preserve subspaces.

\section{Products of Uniform Convergence Spaces}

Besides subspaces, the product uniform convergence structure on
the Cartesian product of a family of uniform convergence spaces is
the simplest example of an initial uniform convergence structure,
and in this section we consider the completions of these initial
uniform convergence spaces.  As is the case for subspaces of
Hausdorff uniform convergence spaces, the Wyler completion of the
product of a family of Hausdorff uniform convergence spaces is in
general different from the product of the completions of the
components.  Indeed, Kent and Ruiz de Eguino \cite{Kent} gave the
following example.
\begin{example}
Consider on $\mathbb{Q}$ the uniform convergence structure induced
by the usual metric, and let $\mathbb{Q}\times \mathbb{Q}$ carry
the product uniform convergence structure.  The completion
$\mathbb{Q}^{\sharp}$ of $\mathbb{Q}$ is $\mathbb{R}$, equipped
with a suitable uniform convergence structure.  In particular, a
filter $\mathcal{F}$ converges to $x$ in $\mathbb{Q}^{\sharp}$ if
and only if
\begin{eqnarray}
[\{\{(x-\epsilon,x+\epsilon)\cup \mathbb{Q}\}\cap \{x\}\mbox{ :
}\epsilon>0\}]\subseteq \mathcal{F}.\nonumber
\end{eqnarray}
On the other hand, the completion $\left(\mathbb{Q}\times
\mathbb{Q}\right)^{\sharp}$ of $\mathbb{Q}\times \mathbb{Q}$ is
$\mathbb{R}\times \mathbb{R}$, equipped with the appropriate
uniform convergence structure.  A filter $\mathcal{G}$ converges
to $\left(x,y\right)$ in $\left(\mathbb{Q}\times
\mathbb{Q}\right)^{\sharp}$ if and only if
\begin{eqnarray}
[\{\{(x-\epsilon,x+\epsilon)\times (y-\epsilon,y+\epsilon)\cap
\mathbb{Q}\times \mathbb{Q}\}\cup \{\left(x,y\right)\}\mbox{ :
}\epsilon>0\}]\subseteq \mathcal{G}.\nonumber
\end{eqnarray}
Consider now the filter
\begin{eqnarray}
\mathcal{H}=[\{\{\left(\frac{1}{n},\pi\right)\mbox{ : }n\geq
k\}\mbox{ : }k\in\mathbb{N}\}].\nonumber
\end{eqnarray}
Clearly the filter $\mathcal{H}$ converges to $\left(0,\pi\right)$
in $\mathbb{Q}^{\sharp}\times \mathbb{Q}^{\sharp}$.  On the other
hand, $\mathcal{H}$ cannot converge to $\left(0,\pi\right)$ in
$\left(\mathbb{Q}\times \mathbb{Q}\right)^{\sharp}$.
\end{example}

The above example shows that the completion of the product of a
family of Hausdorff uniform convergence spaces may differ from the
product of the completions of the components.  However, applying
the results from Section 2 we obtain he following result
concerning the structure of the completion of the product of a
family of uniform convergence spaces.
\begin{theorem}\label{ProdCom}
Let $\left(X_{i}\right)_{i\in I}$ be a family of Hausdorff uniform
convergence spaces, and let $X$ denote their Cartesian product,
equipped with the product uniform convergence structure.  Then
there exists an bijective uniformly continuous mapping
\begin{eqnarray}
\iota_{X}^{\sharp}:X^{\sharp} \rightarrow \prod_{i\in I}X_{i},
\nonumber
\end{eqnarray}
where $X^{\sharp}$ and $X^{\sharp}_{i}$ denote the completions of
$X$ and the $X_{i}$, respectively.
\end{theorem}
\begin{proof}
First note that $\prod_{i\in I}X^{\sharp}_{i}$ is complete
\cite{Wyler}.  For every $i$, let $\iota_{X_{i}}:X_{i}\rightarrow
X^{\sharp}_{i}$ be the uniformly continuous embedding associated
with the completion $X^{\sharp}_{i}$ of $X_{i}$. Define the
mapping $\iota_{X}:X\rightarrow \prod X^{\sharp}_{i}$ through
\begin{eqnarray}\label{VarphiDef}
\iota_{X}:x=\left(x_{i}\right)_{i\in I}\mapsto
\left(\iota_{X_{i}}\left(x_{i}\right)\right)_{i\in I}\nonumber
\end{eqnarray}
For each $i$, let $\pi_{i}:X\rightarrow X_{i}$ be the projection.
Since each $\iota_{X_{i}}$ is injective, so is $\iota_{X}$.
Moreover, we have
\begin{eqnarray}
&\mathcal{U}\in\mathcal{J}_{X}&\Rightarrow\left(\pi_{i}\times
\pi_{i}\right)\left(\mathcal{U}\right)\in\mathcal{J}_{X_{i}}\nonumber\\
&& \Rightarrow\left(\iota_{X_{i}}\times\iota_{X_{i}}\right)
\left(\left(\pi_{i}\times\pi_{i}\right)\left(\mathcal{U}\right)\right)
\in\mathcal{J}^{\sharp}_{X_{i}}\nonumber\\
&&\Rightarrow\prod_{i\in
I}\left(\iota_{X_{i}}\times\iota_{X_{i}}\right)
\left(\left(\pi_{i}\times\pi_{i}\right)\left(\mathcal{U}\right)\right)
\in\mathcal{J}^{\sharp}_{\prod}\nonumber\\
&&\Rightarrow\left(\iota_{X}\times\iota_{X}\right)\left(\mathcal{U}\right)
\in\mathcal{J}^{\sharp}_{\prod}\nonumber
\end{eqnarray}
where $\mathcal{J}_{\prod}^{\sharp}$ denotes the product uniform
convergence structure on $\prod_{i\in I}X_{i}^{\sharp}$.  Hence
$\iota_{X}$ is uniformly continuous. Similarly, if the filter
$\mathcal{V}$ on $\iota_{X}\left(X\right)\times
\iota_{X}\left(X\right)$ belongs to the subspace uniform
convergence structure, then
\begin{eqnarray}
&\left(\pi_{i}\times\pi_{i}\right)\left(\mathcal{V}\right)\in
\mathcal{J}^{\sharp}_{X_{i}}&\Rightarrow\left(\iota_{X_{i}}^{-1}
\times\iota_{X_{i}}^{-1}\right)\left(\left(\pi_{i}\times\pi_{i}\right)
\left(\mathcal{V}\right)\right)\in\mathcal{J}_{X_{i}}\nonumber\\
&&\Rightarrow \prod_{i\in I}\left(\iota_{X_{i}}^{-1}
\times\iota_{X_{i}}^{-1}\right)\left(\left(\pi_{i}\times\pi_{i}\right)
\left(\mathcal{V}\right)\right)\in\mathcal{J}_{X}\nonumber\\
&&\Rightarrow\left(\iota_{X}^{-1}\times\iota_{X}^{-1}\right)
\left(\mathcal{V}\right)\in\mathcal{J}_{X}\nonumber
\end{eqnarray}
so that $\iota_{X}^{-1}$ is uniformly continuous.  Hence
$\iota_{X}$ is a uniformly continuous embedding.  It now follows
by Corollary \ref{EmbExt} that the mapping $\iota_{X}$ extends to
an injective uniformly continuous mapping
\begin{eqnarray}
\iota_{X}^{\sharp}:X^{\sharp} \rightarrow \prod_{i\in
I}X_{i}.\label{IX}
\end{eqnarray}
To see that the mapping (\ref{IX}) is surjective, we note that
$\iota_{X_{i}}\left(X_{i}\right)$ is dense in $X_{i}^{\sharp}$ for
each $i\in I$.  That is,
\begin{eqnarray}
\begin{array}{ll}
\forall & i\in I\mbox{ :} \\
\forall & x_{i}^{\sharp}\in X_{i}^{\sharp}\mbox{ :} \\
\exists & \mathcal{F}_{i}\mbox{ a Cauchy filter on }X_{i}\mbox{ :} \\
& \iota_{X_{i}}\left(\mathcal{F}_{i}\right)\mbox{ converges to }x_{i}^{\sharp}\mbox{ in }X_{i}^{\sharp} \\
\end{array}
\end{eqnarray}
The filter $\mathcal{F} = \prod_{i\in I}\mathcal{F}_{i}$ is a
Cauchy filter on $X$.  As such, there is some $x^{\sharp}\in
X^{\sharp}$ so that $\mathcal{F}$ converges to $x^{\sharp}$ in
$X^{\sharp}$.  Furthermore,
\begin{eqnarray}
\iota_{X}\left(\mathcal{F}\right) = \prod_{i\in
I}\iota_{i}\left(\mathcal{F}_{i}\right)
\end{eqnarray}
so that
$\iota_{X}^{\sharp}\left([\mathcal{F}]_{X^{\sharp}}\right)$
converges to $\left(x_{i}^{\sharp}\right)_{i\in I}$ in
$\prod_{i\in I}X_{i}^{\sharp}$.  As such, it follows by the
uniform continuity of $\iota_{X}^{\sharp}$ that
$\iota_{X}^{\sharp}\left(x^{\sharp}\right) =
\left(x_{i}^{\sharp}\right)_{i\in I}$.  This completes the proof.
\end{proof}

We note here that, similar to the case of the completion of
subspaces of Hausdorff uniform convergence spaces, the Wyler
completion $X^{\sharp}$ of $X= \prod_{i\in I}X_{i}$ is simply the
\textit{set} $\prod_{i\in I}X^{\sharp}_{i}$ equipped with the
finest complete, Hausdorff uniform convergence structure such that
$X$ is contained as a dense subspace.  In particular, we have the
following result, which follows by Theorem \ref{ProdCom} and
Proposition \ref{SubspUCSComp}.
\begin{proposition}\label{ProdProp}
The uniform convergence structure on $X^{\sharp}$ is the smallest
complete, Hausdorff uniform convergence structure on the set
$\prod_{i\in I}X_{i}^{\sharp}$ such that $\prod_{i\in I}X_{i}$ is
contained in $\prod_{i\in I}X_{i}^{\sharp}$ as a dense subspace.
\end{proposition}

\section{Completion of Initial Uniform Convergence Spaces}

In view of the fact that the Wyler completion of uniform
convergence spaces does not, in general, preserve subspaces or
Cartesian products, initial structures are not invariant under the
formation of completions. That is, if $X$ carries the initial
uniform convergence structure with respect to a family of mappings
\begin{eqnarray}
\left(\psi_{i}:X\rightarrow X_{i}\right)_{i\in I}\nonumber
\end{eqnarray}
into Hausdorff uniform convergence spaces $X_{i}$, then the
completion $X^{\sharp}$ of $X$ does not necessarily carry the
initial uniform convergence structure with respect to
\begin{eqnarray}
\left(\psi^{\sharp}_{i}:X^{\sharp}\rightarrow
X^{\sharp}_{i}\right)_{i\in I},\nonumber
\end{eqnarray}
where $\psi^{\sharp}_{i}$ denotes the uniformly continuous
extension of $\psi^{\sharp}_{i}$ to $X^{\sharp}$.  In this regard,
one can at best obtain a generalization of Proposition
\ref{SubspUCSComp} and Theorem \ref{ProdCom}.  The first, and in
fact quite obvious, result in this regard is the following.
\begin{proposition}\label{InitProp}
Suppose that $X$ is equipped with the initial uniform convergence
structure with respect to a family of mappings
\begin{eqnarray}
\left(\varphi_{i}:X\rightarrow X_{i}\right)_{i\in I},\label{Fam}
\end{eqnarray}
where each uniform convergence space $X_{i}$ is Hausdorff, and the
family of mappings (\ref{Fam}) separates the points on $X$.  Then
each mapping $\varphi_{i}$ extends uniquely to a uniformly
continuous mapping
\begin{eqnarray}
\varphi_{i}^{\sharp}:X^{\sharp}\rightarrow
X_{i}^{\sharp}\label{ExtFam}
\end{eqnarray}
and the uniform convergence structure on $X^{\sharp}$ is finer
than the initial uniform convergence structure with respect to the
mappings (\ref{ExtFam}).
\end{proposition}

Concerning the uniform convergence structure on $X^{\sharp}$,
Proposition \ref{InitProp} is, in the general case, the sharpest
result.  However, this result does not give any information
concerning the structure of the \textit{set} $X^{\sharp}$ and its
elements, which is the main interest of this paper. In this
regard, we have the following interesting results.
\begin{lemma}\label{Initlemma}
For each $i\in I$, let $X_{i}$ be a Hausdorff uniform convergence
space, with uniform convergence structure $\mathcal{J}_{X_{i}}$.
Let the uniform convergence space $X$ carry the initial uniform
convergence structure $\mathcal{J}_{X}$ with respect to the family
of mappings
\begin{eqnarray}
\left(\psi_{i}:X\mapsto X_{i}\right)_{i\in I}\nonumber
\end{eqnarray}
Assume that $\left(\psi_{i}\right)_{i\in I}$ separates the points
of $X$.  Then the mapping
\begin{eqnarray}
\Psi:X\ni x\mapsto \left(\psi_{i}\left(x\right)\right)_{i\in
I}\in\prod_{i\in I} X_{i}\label{PsiDef}
\end{eqnarray}
is a uniformly continuous embedding.  In particular, the diagram
\begin{eqnarray}
\begin{array}{c}
\setlength{\unitlength}{1cm} \thicklines
%\begin{pspicture}(13,6)
\begin{picture}(13,6)

\put(1.9,4.9){$X$} \put(2.5,5.0){\vector(1,0){7.4}}
\put(10.1,4.9){$X_{i}$} \put(6.3,5.2){$\psi_{i}$}
\put(2.2,4.6){\vector(1,-1){3.5}} \put(6.5,1.1){\vector(1,1){3.5}}
\put(5.7,0.6){$\prod X_{i}$} \put(3.6,2.5){$\Psi$}
\put(8.5,2.5){$\pi_{i}$}

\end{picture}
%\end{pspicture}
\end{array}\label{PSiDiagram}
\end{eqnarray}
commutes for every $i\in I$.
\end{lemma}
\begin{proof}
Since the family $\left(\varphi_{i}\right)_{i\in I}$ separates the
points of $X$, the mapping (\ref{PsiDef}) is injective.
Furthermore, the diagram (\ref{PSiDiagram}) clearly commutes. Note
that
\begin{eqnarray}
\begin{array}{ll}
\forall & i\in I\mbox{ :} \\
& \left(\psi_{i}\times \psi_{i}\right)\left(\mathcal{U}\right)\in\mathcal{J}_{X_{i}}\mbox{ :} \\
\end{array}.\nonumber
\end{eqnarray}
Hence we have
\begin{eqnarray}
\begin{array}{ll}
\forall & i\in I\mbox{ :} \\
& \left(\pi_{i}\times \pi_{i}\right)\left(\Psi\times \Psi\right)\left(\mathcal{U}\right)\in\mathcal{J}_{X_{i}}\mbox{ :} \\
\end{array}.\nonumber
\end{eqnarray}
Therefore $\left(\Psi\times\Psi\right)\left(\mathcal{U}\right)\in
\mathcal{J}_{\prod}$, which is the product uniform convergence
structure, so that $\Psi$ is uniformly continuous.\\
Let $\mathcal{V}\in\mathcal{J}_{\prod}$ be a filter on
$\prod_{i\in I}X_{i}\times \prod_{i\in I}X_{i}$ with a trace on
$\Psi\left(X\right)\times \Psi\left(X\right)$.  Then
\begin{eqnarray}
\begin{array}{ll}
\forall & i\in I\mbox{ :} \\
& \mbox{a) } \left(\pi_{i}\times\pi_{i}\right)\left(\mathcal{V}\right)\in\mathcal{J}_{X_{i}} \\
& \mbox{b) } W\in\left(\pi_{i}\times\pi_{i}\right)\left(\mathcal{V}\right) \Rightarrow W\cap\left(\psi_{i}\left(X\right)\times \psi_{i}\left(X\right)\right)\neq\emptyset \\
\end{array}
\end{eqnarray}
so that
\begin{eqnarray}
\begin{array}{ll}
\forall & i\in I\mbox{ :} \\
& \left(\psi_{i}\times\psi_{i}\right)\left(\left(\Psi^{-1}\times\Psi^{-1}\right)\left(\mathcal{V}\right)\right)\supseteq \left(\pi_{i}\times\pi_{i}\right)\left(\mathcal{V}\right) \\
\end{array}\nonumber
\end{eqnarray}
Form the definition of an initial uniform convergence structure,
and in particular the product uniform convergence structure, it
follows that $\left(\Psi^{-1}\times
\Psi^{-1}\right)\left(\mathcal{V}\right)\in\mathcal{J}_{X}$. Hence
$\Psi$ is a uniformly continuous embedding.
\end{proof}
The following is now a straight forward application of Lemma
\ref{Initlemma}, Theorem \ref{ProdCom} and Proposition
\ref{SubspUCSComp}.
\begin{theorem}\label{IUCSComp}
For each $i\in I$, let $X_{i}$ be a Hausdorff uniform convergence
space, with uniform convergence structure $\mathcal{J}_{X_{i}}$.
Let the uniform convergence space $X$ carry the initial uniform
convergence structure $\mathcal{J}_{X}$ with respect to the family
of mappings
\begin{eqnarray}
\left(\psi_{i}:X\mapsto X_{i}\right)_{i\in I}\nonumber
\end{eqnarray}
Assume that $\left(\psi_{i}\right)_{i\in I}$ separates the points
of $X$.  Then there exists a unique injective, uniformly
continuous mapping
\begin{eqnarray}
\Psi^{\sharp}:X^{\sharp}\rightarrow \prod_{i\in
I}X^{\sharp}_{i}\label{Embedding}
\end{eqnarray}
such that, for each $i\in I$, the diagram
\begin{eqnarray}
\begin{array}{c}
\setlength{\unitlength}{1cm} \thicklines
%\begin{pspicture}(13,6)
\begin{picture}(13,6)

\put(1.9,4.9){$X^{\sharp}$} \put(2.5,5.0){\vector(1,0){7.4}}
\put(10.1,4.9){$X^{\sharp}_{i}$}
\put(6.3,5.4){$\psi_{i}^{\sharp}$}
\put(2.2,4.6){\vector(1,-1){3.5}} \put(6.5,1.1){\vector(1,1){3.5}}
\put(5.7,0.6){$\prod X_{i}^{\sharp}$}
\put(3.4,2.5){$\Psi^{\sharp}$} \put(8.5,2.5){$\pi_{i}$}

\end{picture} \\
\end{array}
\end{eqnarray}
commutes, with $\pi_{i}$ the projection, and $\psi_{i}^{\sharp}$
the unique extension of $\psi_{i}$ to $X^{\sharp}$.
\end{theorem}

Within the context of nonlinear PDEs, as explained in Section 1,
Theorem \ref{IUCSComp} may be interpreted as a regularity result.
Indeed, consider some space $X\subseteq \mathcal{C}^{\infty}
\left(\Omega\right)$ of classical, smooth functions on an open,
nonempty subset $\Omega$ of $\mathbb{R}^{n}$.  Equip $X$ with the
initial uniform convergence structure $\mathcal{J}_{X}$ with
respect to the family of mappings
\begin{eqnarray}
D^{\alpha}:X\rightarrow Y\mbox{,
}\alpha\in\mathbb{N}^{n}\label{DerMaps}
\end{eqnarray}
where $Y$ is some space of functions on $\Omega$ that contains
$D^{\alpha}\left(X\right)$ for each $\alpha\in\mathbb{N}^{n}$.  In
view of Lemma \ref{Initlemma}, the mapping
\begin{eqnarray}
\textbf{D}:X\ni u\rightarrow \left(D^{\alpha}u\right) \in
Y^{\mathbb{N}},\label{D}
\end{eqnarray}
while Theorem \ref{IUCSComp} guarantees that the extension
\begin{eqnarray}
\textbf{D}^{\sharp}:X^{\sharp}\ni u\rightarrow
\left(D^{\alpha}u\right) \in Y^{\sharp\mathbb{N}}\label{ExtD}
\end{eqnarray}
 of (\ref{D}) is injective and that the diagram
\begin{eqnarray}
\begin{array}{c}
\setlength{\unitlength}{1cm} \thicklines
%\begin{pspicture}(13,6)
\begin{picture}(13,6)

\put(1.9,4.9){$X^{\sharp}$} \put(2.5,5.0){\vector(1,0){7.4}}
\put(10.1,4.9){$Y^{\sharp}$} \put(6.3,5.2){$D^{\alpha\sharp}$}
\put(2.2,4.6){\vector(1,-1){3.5}} \put(6.5,1.1){\vector(1,1){3.5}}
\put(5.7,0.6){$Y^{\sharp\mathbb{N}}$}
\put(3.2,2.5){$\textbf{D}^{\sharp}$} \put(8.5,2.5){$\pi_{\alpha}$}

\end{picture}
%\end{pspicture}
\end{array}\label{DDiagram}
\end{eqnarray}
commutes.  Here
\begin{eqnarray}
D^{\alpha\sharp}:X^{\sharp}\rightarrow Y^{\sharp}\mbox{, }\alpha
\in \mathbb{N}^{n}
\end{eqnarray}
are the uniformly continuous extension of the mappings
(\ref{DerMaps}).  As such, each \textit{generalized function}
$u^{\sharp}\in X^{\sharp}$ may be identified with
$\textbf{D}^{\sharp}u^{\sharp} \in Y^{\sharp\mathbb{N}}$.

\section{An Application to Nonlinear PDEs}

The Order Completion Method \cite{Oberguggenberger} for nonlinear
partial differential equations is a general and type independent
theory for the existence and regularity of generalized solutions
of nonlinear PDEs.  The generalized solutions obtained through
this method are constructed as elements of the Dedekind completion
of suitable spaces of piecewise smooth functions. Recently, see
\cite{vdWalt6} through \cite{vdWalt7}, this method was
significantly enriched by reformulating it in terms of suitable
uniform convergence structures, notably the uniform order
convergence structure \cite{vdWalt6}.

We now present, as an application of the results obtained in
Sections 2, 3 and 4, a further enrichment of the basic theory.  In
particular, we prove existence and basic regularity results for
generalized solutions of arbitrary $\mathcal{C}^{k}$-smooth
systems of nonlinear PDEs. In this regard, consider a system of
$K$ nonlinear PDEs, each of order at most $m$, of the form
\begin{eqnarray}
\textbf{T}\left(x,D\right)\textbf{u}\left(x\right)=\textbf{f}\left(x\right)\mbox{,
}x\in\Omega,\label{OCMNLPDE}
\end{eqnarray}
where $\Omega\subseteq\mathbb{R}^{n} $ is some nonempty open
subset of $\textbf{R}^{n}$.  The right hand term
$\textbf{f}:\Omega\rightarrow \mathbb{R}^{K}$ is assumed to be a
$\mathcal{C}^{k}$-smooth function on $\Omega$, with components
$f_{1},...,f_{K}$, while the partial differential operator
$\textbf{T}\left(x,D\right)$ is supposed to be defined through a
$\mathcal{C}^{k}$-smooth function
\begin{eqnarray}
\textbf{F}:\Omega\times\mathbb{R}^{M}\rightarrow
\mathbb{R}^{K}\label{PDEDefMap}
\end{eqnarray}
by
\begin{eqnarray}
\begin{array}{ll}
\forall & \textbf{u}\in\mathcal{C}^{k}\left(\Omega\right)^{K}\mbox{ :} \\
\forall & x\in\Omega\mbox{ :} \\
& \textbf{T}\left(x,D\right)u\left(x\right)= F\left(x,u\left(x\right),...,D^{\alpha}u\left(x\right),...\right),\mbox{, }|\alpha|\leq m \\
\end{array},\label{PDEOPDef}
\end{eqnarray}
with $k\in\mathbb{N}\cup\{\infty\}$.  We also make the following
technical assumption:
\begin{eqnarray}
\begin{array}{ll}
\forall & x\in\Omega\mbox{ :} \\
& \textbf{f}\left(x\right)\in\textrm{int}\{\textbf{F}\left(x,\xi\right)\mbox{ : }\xi\in\mathbb{R}^{M}\} \\
\end{array}\label{Ass}
\end{eqnarray}
Note that (\ref{Ass}) is merely a a necessary condition for the
existence of a classical solution to (\ref{OCMNLPDE}) on a
neighborhood of $x\in \Omega$.

We construct generalized solutions to (\ref{OCMNLPDE}) which may
be \textit{assimilated} with functions which are
$\mathcal{C}^{k}$-smooth everywhere on $\Omega$, except possible
on a closed nowhere dense set $\Gamma\subset\Omega$.  In this
regard, we consider the space
$\mathcal{NL}\left(\Omega\right)^{K}$ of nearly finite, normal
lower semi-continuous functions on $\Omega$.  Recall
\cite{Dilworth}, see also \cite{vdWalt6}, that an extended real
valued function
\begin{eqnarray}
u:\Omega\rightarrow \overline{\mathbb{R}} \nonumber
\end{eqnarray}
is normal lower semi-continuous at $x\in\Omega$ whenever
\begin{eqnarray}
\left(I\circ S\right)\left(u\right)\left(x\right) =
u\left(x\right),\label{NLSDef}
\end{eqnarray}
and $u$ is normal lower semi-continuous on $\Omega$ whenever it is
normal lower semi-continuous at every $x\in\Omega$.  Here $I$ and
$S$ are the Lower and Upper Baire Operators, respectively, defined
through
\begin{eqnarray}
I\left(u\right)\left(x\right) = \sup\{\inf\{u\left(y\right)\mbox{
: }\|x-y\|<\delta\}\mbox{ : }\delta>0\}\mbox{, }x\in
\Omega\label{IDef}
\end{eqnarray}
and
\begin{eqnarray}
S\left(u\right)\left(x\right) = \inf\{\sup\{u\left(y\right)\mbox{
: }\|x-y\|<\delta\}\mbox{ : }\delta>0\}\mbox{, }x\in
\Omega.\label{SDef}
\end{eqnarray}
It is clear that an extended real valued mapping on $\Omega$ is
normal lower semi-continuous at $x\in\Omega$ whenever $u$ is real
valued and continuous at $x$.  A normal lower semi-continuous
function is called nearly finite whenever the set
\begin{eqnarray}
\left\{x\in\Omega\mbox{ : }u\left(x\right)\in\mathbb{R} \right\}
\end{eqnarray}
is open and dense in $\Omega$.  The set $\mathcal{NL}\left(\Omega
\right)$ is a fully distributive and Dedekind complete lattice
with respect to the pointwise order
\begin{eqnarray}
u\leq v \Leftrightarrow \left(\begin{array}{ll}
\forall & x\in X\mbox{ :} \\
& u\left(x\right)\leq v\left(x\right) \\
\end{array}\right).\label{NLOrder}
\end{eqnarray}
We consider the following subspaces of
$\mathcal{NL}\left(\Omega\right)$.  Namely, for
$l\in\mathbb{N}\cup \{\infty\}$, we consider the set
\begin{eqnarray}
\mathcal{ML}^{l}\left(\Omega\right) = \left\{u\in\mathcal{NL}
\left(\Omega\right)\mbox{ }\begin{array}{|ll}
\exists & \Gamma\subset\Omega\mbox{ closed nowhere dense :} \\
& u\in\mathcal{C}^{l}\left(\Omega\setminus\Gamma\right) \\
\end{array}\right\}.\label{MLlDef}
\end{eqnarray}
\begin{theorem}\label{MlLattice}
For each $l\geq 0$, the space
$\mathcal{ML}^{l}\left(\Omega\right)$ is a fully distributive
lattice with respect to the pointwise order (\ref{NLOrder}).
\end{theorem}
\begin{proof}
Consider any $u,v\in \mathcal{ML}^{l}\left(\Omega\right)$.  Then
there is a closed and nowhere dense subset $\Gamma$ of $\Omega$
such that $u,v \in \mathcal{C}^{l}\left(\Omega\setminus\Gamma
\right)$.  Define the open subsets $U$, $V$ and $W$ of $\Omega
\setminus\Gamma$ through
\begin{eqnarray}
U=\{x\in\Omega\setminus\Gamma\mbox{ : }
u\left(x\right)<v\left(x\right) \},\nonumber
\end{eqnarray}
\begin{eqnarray}
V=\{x\in\Omega\setminus\Gamma\mbox{ : }
v\left(x\right)<u\left(x\right) \}\nonumber
\end{eqnarray}
and
\begin{eqnarray}
W=\textrm{int}\{x\in\Omega\setminus\Gamma\mbox{ : }
u\left(x\right)=v\left(x\right) \},\nonumber
\end{eqnarray}
It is clear that the function
\begin{eqnarray}
\varphi:\Omega\ni x\mapsto
\sup\{u\left(x\right),v\left(x\right)\}\in\overline{\mathbb{R}}\nonumber
\end{eqnarray}
is $\mathcal{C}^{l}$-smooth on $U\cup V\cup W$.  Furthermore, the
set $U\cup V\cup W$ is dense in $\Omega\setminus\Gamma$.  As such,
it follows by \cite[Theorem 1]{vdWalt6} that $\sup\{u,v\}$ belongs
to
$\mathcal{ML}^{l}\left(\Omega\right)$.\\
The existence of the infimum of $u,v\in\mathcal{ML}^{l} \left(
\Omega\right)$ in $\mathcal{ML}^{l} \left( \Omega\right)$ follows
in the same way.  The distributivity of
$\mathcal{ML}^{l}\left(\Omega\right)$ now follows by the
corresponding property for $\mathcal{NL}\left(\Omega\right)$.
\end{proof}

With the nonlinear partial differential operator (\ref{PDEOPDef})
we may associate a mapping
\begin{eqnarray}
\textbf{T}:\mathcal{ML}^{m+k}\left(\Omega\right)^{K} \rightarrow
\mathcal{ML}^{k}\left(\Omega\right)^{K}.\label{ExtPDEMap}
\end{eqnarray}
In particular, the components of the mapping (\ref{ExtPDEMap}) may
be defined through
\begin{eqnarray}
T_{j}:\mathcal{ML}^{m+k}\left(\Omega\right)^{K} \ni \textbf{u}
\mapsto \left(I\circ
S\right)\left(F_{j}\left(\cdot,...,\mathcal{D}^{\alpha}u_{i},...\right)\right)
\in \mathcal{ML}^{k}\left(\Omega\right)
\end{eqnarray}
where, for $j=1,...,K$, the mappings $F_{j}:\Omega\times
\mathbb{R}^{M}\rightarrow \mathbb{R}$ are the components of
(\ref{PDEDefMap}).  On the space
$\mathcal{ML}^{m+k}\left(\Omega\right)^{K}$ we define the
equivalence relation
\begin{eqnarray}
\begin{array}{ll}
\forall & \textbf{u},\textbf{v}\in\mathcal{ML}^{m+k}\left(\Omega\right)^{K}\mbox{ :} \\
& \textbf{u}\sim_{\textbf{T}}\textbf{v}\Leftrightarrow \textbf{Tu}=\textbf{Tv} \\
\end{array},\label{TEq}
\end{eqnarray}
and we denote the quotient space
$\mathcal{ML}^{m+k}\left(\Omega\right)^{K}/\sim_{\textbf{T}}$ by
$\mathcal{ML}^{k}_{T}\left(\Omega\right)$.  We may associate with
the mapping $\textbf{T}$ an injective mapping
\begin{eqnarray}
\widehat{\textbf{T}}:\mathcal{ML}^{m+k}_{\textbf{T}}\left(\Omega\right)
\rightarrow \mathcal{ML}^{k}\left(\Omega\right)^{K}
\end{eqnarray}
so that the diagram
\begin{eqnarray}
\begin{array}{c}
\setlength{\unitlength}{1cm} \thicklines
%\begin{pspicture}(13,6)
\begin{picture}(13,6)

\put(1.9,5.2){$\mathcal{ML}^{m+k}\left(\Omega\right)^{K}$}
\put(4.3,5.3){\vector(1,0){4.6}}
\put(9.1,5.2){$\mathcal{ML}^{k}\left(\Omega\right)^{K}$}
\put(5.7,5.5){$\textbf{T}$} \put(2.2,5.0){\vector(0,-1){3.5}}
\put(3.7,1.0){\vector(1,0){5.2}}
\put(1.9,0.9){$\mathcal{ML}^{m+k}_{\textbf{T}}\left(\Omega\right)$}
\put(9.1,0.9){$\mathcal{ML}^{k}\left(\Omega\right)^{K}$}
\put(1.6,3.2){$q_{\textbf{T}}$} \put(10.4,3.2){$i$}
\put(10.2,5.0){\vector(0,-1){3.5}}
\put(5.7,1.2){$\widehat{\textbf{T}}$}

\end{picture}
%\end{pspicture}
\end{array}\label{TTHat}
\end{eqnarray}
commutes, with $q_{\textbf{T}}$ the canonical quotient mapping
associated with the equivalence relation (\ref{TEq}), and $i$ the
identity.  In view of the commutative diagram (\ref{TTHat}), the
equation
\begin{eqnarray}
\textbf{Tu}=\textbf{f},\label{TPDE}
\end{eqnarray}
which is an extension of the system of PDEs (\ref{OCMNLPDE}), is
equivalent to
\begin{eqnarray}
\widehat{\textbf{T}}\textbf{U}=\textbf{f}\label{THatPDE}
\end{eqnarray}
in the sense that
\begin{eqnarray}
\begin{array}{ll}
\forall & \textbf{u}\in \mathcal{ML}^{m+k}\left(\Omega\right)^{K}\mbox{ :} \\
& \textbf{Tu}=\textbf{f}\Leftrightarrow \widehat{\textbf{T}}\left(q_{\textbf{T}}\left(\textbf{u}\right)\right)=\textbf{f} \\
\end{array}
\end{eqnarray}
and
\begin{eqnarray}
\begin{array}{ll}
\forall & \textbf{U}\in \mathcal{ML}^{m+k}_{\textbf{T}}\left(\Omega\right)\mbox{ :} \\
& \widehat{\textbf{T}}\textbf{U}=\textbf{f}\Leftrightarrow \textbf{T}\left(q_{\textbf{T}}^{-1}\left(\textbf{U}\right)\right)=\{\textbf{f}\} \\
\end{array}.
\end{eqnarray}

Let us now introduce suitable uniform convergence structures on
$\mathcal{ML}^{m+k}\left(\Omega\right)^{K}$ and
$\mathcal{ML}^{m+k}_{\textbf{T}}\left(\Omega\right)$.  In this
regard, recall \cite{RieszI} that a sequence $\left(x_{n}\right)$
on a partially ordered set $L$ order converges to $x\in L$ if and
only if
\begin{eqnarray}
\begin{array}{ll}
\exists & \left(\lambda_{n}\right)\mbox{, }\left(\mu_{n}\right)\subset L\mbox{ :} \\
& \begin{array}{ll} 1) & n\in\mathbb{N}\Rightarrow
\left(\begin{array}{ll}
1.1) & \lambda_{n}\leq \lambda_{n+1}\leq \mu_{n+1}\leq \mu_{n} \\
1.2) & \lambda_{n}\leq u_{n}\leq \mu_{n} \\
\end{array}\right) \\
2) & \sup\{\lambda_{n}\mbox{ : }n\in\mathbb{N}\}=u=\inf\{\mu_{n}\mbox{ : }n\in\mathbb{N}\} \\
\end{array} \\
\end{array}.\label{OCDef}
\end{eqnarray}
In general, the order convergence of sequences cannot be induced
through a topology.  That is, for a partially ordered set $L$,
there is in general no topology $\tau$ on $L$ so that a sequence
$\left(x_{n}\right)$ on $L$ converges to $x\in L$ with respect to
$\tau$ if and only if $\left(x_{n}\right)$ order converges to $x$.

However, in view of Theorem \ref{MlLattice}, the order convergence
of sequences on $\mathcal{ML}^{k}\left(\Omega\right)$ may be
induced by a convergence structure \cite{Anguelov and van der
Walt}.  In particular, the order convergence structure
$\lambda_{o}$, which is defined through
\begin{eqnarray}
\mathcal{F}\in\lambda_{o}\left(u\right)\Leftrightarrow
\left(\begin{array}{ll}
\forall & n\in \mathbb{N}\mbox{ :} \\
\exists & [\lambda_{n},\mu_{n}]\subset \mathcal{ML}^{l}\left(\Omega\right)\mbox{ :} \\
& \begin{array}{ll}
1) & n\in\mathbb{N}\Rightarrow [\lambda_{n+1},\mu_{n+1}]\subseteq [\lambda_{n},\mu_{n}] \\
2) & \sup\{\lambda_{n}\mbox{ : }n\in\mathbb{N}\}=u=\inf\{\mu_{n}\mbox{ : }n\in\mathbb{N}\} \\
3) & [\{[\lambda_{n},\mu_{n}]\mbox{ : }n\in\mathbb{N}\}]\subseteq\mathcal{F} \\
\end{array} \\
\end{array}\right),\label{OCSDef}
\end{eqnarray}
induces the order convergence of sequences.  Furthermore,
$\lambda_{o}$ is Hausdorff and first countable, see \cite[Theorem
17]{Anguelov and van der Walt}.  The Cartesian product
$\mathcal{ML}^{k}\left(\Omega\right)^{K}$ is equipped with the
product convergence structure $\lambda_{o}^{K}$, see
\cite{Beattie}, which is defined through
\begin{eqnarray}
\mathcal{F}\in\lambda_{o}^{K}\left(\textbf{u}\right)
\Leftrightarrow \left(\begin{array}{ll}
\forall & i=1,...,K\mbox{ :} \\
& \pi_{i}\left(\mathcal{F}\right)\in\lambda_{o}\left(u_{i}\right) \\
\end{array}\right),\label{ProdCStr}
\end{eqnarray}
where $\pi_{i}$ denotes the projection.  Since $\lambda_{o}$ is
Hausdorff and first countable, so is $\lambda_{o}^{K}$.
Furthermore, a sequence $\left(\textbf{u}_{n}\right)$ in
$\mathcal{ML}^{k}\left(\Omega\right)^{K}$ converges to $\textbf{u}
\in \mathcal{ML}^{k}\left(\Omega\right)^{K}$ if and only if
$\left(u_{n,i}\right)$ order converges to $u_{i}$ for each
$i=1,...,K$.

In view of the fact that $\lambda_{o}^{K}$ is Hausdorff, it
follows that the associated uniform convergence structure
$\mathcal{J}_{\lambda_{o}^{K}}$, \cite[Proposition
2.1.7]{Beattie}, namely
\begin{eqnarray}
\begin{array}{ll}
\forall & \mathcal{U}\mbox{ a filter on $\mathcal{ML}^{k}\left(\Omega\right)^{K}\times \mathcal{ML}^{k}\left(\Omega\right)^{K}$ :} \\
& \mathcal{U}\in\mathcal{J}_{\lambda_{o}^{K}}\Leftrightarrow
\left(\begin{array}{ll}
\exists & \mathcal{F}_{1},...,\mathcal{F}_{k}\mbox{ filters on $\mathcal{ML}^{k}\left(\Omega\right)^{K}$ :} \\
\exists & \textbf{u}_{1},...,\textbf{u}_{k}\in\mathcal{ML}^{k}\left(\Omega\right)^{K}\mbox{ :} \\
& \begin{array}{ll}
1) & \mathcal{F}_{i}\in\lambda_{o}\left(u_{i}\right)\mbox{, }i=1,...,K \\
2) & \mathcal{U}\supseteq \left(\mathcal{F}_{1}\times\mathcal{F}_{1}\right)\cap...\cap\left(\mathcal{F}_{k}\times\mathcal{F}_{k}\right) \\
\end{array} \\
\end{array}\right) \\
\end{array}\label{AssUCS}
\end{eqnarray}
is uniformly Hausdorff and complete.  Furthermore, the uniform
convergence structure (\ref{AssUCS}) induces the convergence
structure $\lambda_{o}^{K}$ on $\mathcal{ML}^{k}\left(\Omega
\right)^{K}$.

The space $\mathcal{ML}^{m+k}_{\textbf{T}}\left(\Omega\right)$
will carry the initial uniform convergence structure
$\mathcal{J}_{\widehat{\textbf{T}}}$ with respect to the injective
mapping
\begin{eqnarray}
\widehat{\textbf{T}}:\mathcal{ML}^{m+k}_{\textbf{T}}\left(\Omega\right)
\rightarrow \mathcal{ML}^{k}\left(\Omega\right)^{K}.\nonumber
\end{eqnarray}
That is,
\begin{eqnarray}
\mathcal{U}\in\mathcal{J}_{\widehat{\textbf{T}}} \Leftrightarrow
\left(\widehat{\textbf{T}}\times
\widehat{\textbf{T}}\right)\left(\mathcal{U}\right) \in
\mathcal{J}_{\lambda_{o}^{K}}. \label{JTDef}
\end{eqnarray}
The following is now immediate.
\begin{proposition}
The mapping $\widehat{\textbf{T}}$ is a uniformly continuous
embedding of the uniform convergence space
$\mathcal{M}^{k}_{\textbf{T}}\left(\Omega\right)$ into the uniform
convergence space $\mathcal{ML}^{k}\left(\Omega\right)^{K}$.
\end{proposition}
As in the rest of the paper, we denote by
$\mathcal{ML}^{m+k}_{\textbf{T}}\left(\Omega\right)^{\sharp}$ the
uniform convergence space completion of
$\mathcal{ML}^{m+k}_{\textbf{T}}\left(\Omega\right)$. The
extension of the uniformly continuous embedding
\begin{eqnarray}
\widehat{\textbf{T}}:\mathcal{ML}^{m+k}_{\textbf{T}}\left(\Omega\right)
\rightarrow \mathcal{ML}^{k}\left(\Omega\right)^{K}\nonumber
\end{eqnarray}
is denoted $\widehat{\textbf{T}}^{\sharp}$.  The
\textit{generalized} equation, corresponding to (\ref{OCMNLPDE}),
now takes the form
\begin{eqnarray}
\widehat{\textbf{T}}^{\sharp}\textbf{U}^{\sharp}=\textbf{f}\label{ExtPDEII}
\end{eqnarray}
In view of the equivalence of the equations (\ref{TPDE}) and
(\ref{THatPDE}), a solution $U^{\sharp}$ of (\ref{ExtPDEII}) is
interpreted as \textit{generalized solution} of (\ref{OCMNLPDE}).

The existence of solutions of (\ref{ExtPDEII}) follows as an
application of the following basic approximation result
\cite{vdWalt2}.
\begin{theorem}\label{Approx}
Consider a system of nonlinear PDEs of the form (\ref{OCMNLPDE})
through (\ref{PDEOPDef}) that also satisfies (\ref{Ass}).  For
every $\epsilon>0$ there exists a closed nowhere dense set
$\Gamma_{\epsilon}\subset\Omega$, and a function
$\textbf{u}_{\epsilon}\in\mathcal{C}^{\infty}\left(\Omega\setminus\Gamma_{\epsilon}\right)^{K}$
such that
\begin{equation}\label{ApEq}
f_{j}\left(x\right)-\epsilon\leq
T_{j}\left(x,D\right)u_{\epsilon}\left(x\right)\leq
f_{j}\left(x\right)\mbox{, }x\in\Omega\setminus\Gamma_{\epsilon}
\end{equation}
for every $j=1,...,K$.
\end{theorem}
The main result of this section is the following.
\begin{theorem}\label{Exist}
Consider a system of nonlinear PDEs of the form (\ref{OCMNLPDE})
through (\ref{PDEOPDef}).  For every
$\textbf{f}\in\mathcal{C}^{k}\left(\Omega\right)^{K}$ that
satisfies (\ref{Ass}), there exists a unique
$\textbf{U}^{\sharp}\in\mathcal{ML}^{m+k}_{\textbf{T}}\left(\Omega\right)^{\sharp}$
such that
\begin{eqnarray}
\widehat{\textbf{T}}^{\sharp}\textbf{U}^{\sharp}=\textbf{f}.\nonumber
\end{eqnarray}
\end{theorem}
\begin{proof}
First let us show existence.  For every $n\in\textbf{N}$, Theorem
\ref{Approx} yields a closed nowhere dense set
$\Gamma_{n}\subset\Omega$ and a function
$\textbf{u}_{n}\in\mathcal{C}^{\infty}\left(\Omega\setminus\Gamma_{n}\right)^{K}$
that satisfies
\begin{equation}
x\in\Omega\setminus\Gamma_{n}\Rightarrow
f_{j}\left(x\right)-\frac{1}{n}\leq
T_{j}\left(x,D\right)\textbf{u}_{n}\left(x\right)\leq
f_{j}\left(x\right)\label{Approx2}
\end{equation}
for every $j=1,...,K$.  Since $\Gamma_{n}$ is closed nowhere dense
we associate $\textbf{u}_{n}$ with the function $\textbf{v}_{n}\in
\mathcal{ML}^{m+k} \left(\Omega\right)^{K}$, the components of
which are defined as
\begin{eqnarray}
v_{n,i} = \left(I\circ S\right)\left(u_{n,i}\right).\nonumber
\end{eqnarray}
Clearly, we now have, for each $n\in\mathbb{N}$ and $j=1,...,K$,
the inequalities
\begin{eqnarray}
f_{j}-\frac{1}{n}\leq T_{j}\textbf{v}_{n}\leq f_{j}\label{Approx3}
\end{eqnarray}
Consider now, for each $n\in\mathbb{N}$, the equivalence class
$\textbf{U}_{n}\in\mathcal{ML}^{m+k}_{\textbf{T}}\left(\Omega\right)$
associated
with the function $\textbf{v}_{n}$ through (\ref{TEq}).\\
Clearly, the sequence
$\left(\widehat{\textbf{T}}\textbf{U}_{n}\right)$ converges to
$\textbf{f}$ in $\mathcal{ML}^{k}\left(\Omega\right)^{K}$.  It now
follows that $\left(\textbf{U}_{n}\right)$ is a Cauchy sequence in
$\mathcal{ML}^{m+k}_{\textbf{T}}\left(\Omega\right)$ so that there
exists
$\textbf{U}^{\sharp}\in\mathcal{ML}^{m+k}_{_{\textbf{T}}}\left(\Omega\right)$
that satisfies (\ref{ExtPDEII}).\\
Since the mapping
$\widehat{\textbf{T}}:\mathcal{ML}^{m+k}_{\textbf{T}}\left(\Omega\right)
\rightarrow\mathcal{ML}^{k}\left(\Omega\right)^{K}$ is a uniformly
continuous embedding, the uniqueness of the solution
$\textbf{U}^{\sharp}$ found above now follows by Corollary
\ref{EmbExt}.
\end{proof}

Note that the \textit{uniqueness} of the generalized solution
$\textbf{U}^{\sharp}$ to (\ref{ExtPDEII}) should not be
misinterpreted as implying that any, possibly classical, solutions
are disregarded. In fact, quite the contrary.  Recall that the
completion of a uniform convergence space $X$ may be obtained
\textit{constructively}. In particular, it consists of all
equivalence classes of Cauchy filters on $X$ so that the members
of an equivalence class $[\mathcal{F}]$ all converge to the same
element of the completion $X^{\sharp}$.  That is, if we denote by
$X_{C}$ the set of Cauchy filters on $X$, then
\begin{eqnarray}
X^{\sharp}= X_{C}/\sim_{C}
\end{eqnarray}
where $\sim_{C}$ is the equivalence relation on $X_{C}$ defined as
\begin{eqnarray}
\mathcal{F}\sim_{C}\mathcal{G} \Leftrightarrow
\mathcal{F}\cap\mathcal{G}\in X_{C}.
\end{eqnarray}
In view of this, the \textit{unique} generalized solution is in
fact the \textit{totality of all approximate solutions in}
$\mathcal{ML}^{m+k}\left(\Omega\right)^{K}$.  In particular, every
classical solution $\textbf{u}\in\mathcal{C}^{m+k}\left(\Omega
\right)^{K}$ of (\ref{OCMNLPDE}), and every solution
$\textbf{u}\in\mathcal{ML}^{m+k}\left(\Omega\right)^{K}$ of
(\ref{TPDE}) generates a Cauchy filter on
$\mathcal{ML}^{m+k}_{\textbf{T}}\left(\Omega\right)$ which
converges to $\textbf{U}^{\sharp}$ in
$\mathcal{ML}^{m+k}_{\textbf{T}} \left(\Omega\right)^{\sharp}$. As
such, the unique generalized solution of (\ref{OCMNLPDE}) contains
also all of the mentioned usual solutions, should such solutions
exist.

Notice also that the mapping
 \begin{eqnarray}
\widehat{\textbf{T}}^{\sharp}:\mathcal{ML}^{m+k}_{\textbf{T}}\left(\Omega\right)^{\sharp}\rightarrow
\mathcal{ML}^{k}\left(\Omega\right)^{K}\nonumber
\end{eqnarray}
is \textit{injective}.  As such, we may consider the completion
$\mathcal{ML}^{m+k}_{\textbf{T}}\left(\Omega\right)^{\sharp}$ of
$\mathcal{ML}^{m+k}_{\textbf{T}}\left(\Omega\right)$ as a subset
of $\mathcal{ML}^{k}\left(\Omega\right)^{K}$, equipped with a
suitable uniform convergence structure.  Hence, as a bonus, we
also have a \textit{blanket regularity} in the sense that every
element $\textbf{U}^{\sharp}$ of
$\mathcal{ML}^{m+k}_{\textbf{T}}\left(\Omega\right) ^{\sharp}$ may
be assimilated with elements of $\mathcal{ML}^{k} \left(\Omega
\right)^{K}$.

The results on existence, uniqueness and regularity of generalized
solutions to (\ref{OCMNLPDE}) obtained in this section are, to a
certain extent, maximal with respect to the regularity of
solutions within the framework of the so called pullback spaces of
generalized functions considered here. In this regard, let us now
formulate the construction of generalized solution in an abstract
framework. Consider spaces $X$ and $Y$ of functions
$\textbf{g}:\Omega\rightarrow \mathbb{R}^{K}$ such that
$\textbf{f}\in Y$, and the nonlinear partial differential operator
$\textbf{T}$ associated with (\ref{OCMNLPDE}) acts as
\begin{eqnarray}
\textbf{T}:X\rightarrow Y.\label{T}
\end{eqnarray}
Also suppose that $Y$ is equipped with a complete and Hausdorff
uniform convergence structure $\mathcal{J}_{Y}$ which is first
countable. Proceeding in the same way as is done in this section,
we introduce an equivalence relation on $X$ through
\begin{eqnarray}
\textbf{u}\sim_{\textbf{T}}\textbf{v}\Leftrightarrow \textbf{Tu} =
\textbf{Tv},\nonumber
\end{eqnarray}
and associate with the mapping (\ref{T}) the injective mapping
\begin{eqnarray}
\widehat{\textbf{T}}_{X}:X_{\textbf{T}}\rightarrow Y,\label{THat}
\end{eqnarray}
where $X_{\textbf{T}}$ is the quotient space
$X/\sim_{\textbf{T}}$. In particular, the mapping (\ref{THat}) is
supposed to satisfy
\begin{eqnarray}
\begin{array}{ll}
\forall & \textbf{U}\in X_{\textbf{T}}\mbox{ :} \\
\forall & \textbf{u}\in\textbf{U}\mbox{ :} \\
& \textbf{Tu}=\widehat{\textbf{T}}_{X}\textbf{U}=\textbf{f} \\
\end{array}.\nonumber
\end{eqnarray}
If we equip $X_{\textbf{T}}$ with the initial uniform convergence
structure $\mathcal{J}_{\textbf{T}}$ with respect to the mapping
(\ref{THat}), then $\mathcal{J}_{\textbf{T}}$ is Hausdorff and
first countable.  In particular, the mapping (\ref{THat}) is a
uniformly continuous embedding, and extends uniquely to a
injective uniformly continuous mapping
\begin{eqnarray}
\widehat{\textbf{T}}^{\sharp}_{X}:X_{\textbf{T}}^{\sharp}\rightarrow
Y,\label{ExtTHat}
\end{eqnarray}
where $X_{\textbf{T}}^{\sharp}$ is the completion of
$X_{\textbf{T}}$.  A generalized solution of the systems of
nonlinear PDEs
\begin{eqnarray}
\textbf{Tu}=\textbf{f}\nonumber
\end{eqnarray}
in this context is any solution $\textbf{U}^{\sharp}\in
X_{\textbf{T} }^{\sharp}$ of the equation
\begin{eqnarray}
\widehat{\textbf{T}}^{\sharp}_{X}\textbf{U}^{\sharp}=\textbf{f}.\label{TEquation}
\end{eqnarray}
Note that, in view of the fact that the mapping (\ref{THat}) is a
uniformly continuous embedding, and (\ref{ExtTHat}) therefore an
injection, the equation (\ref{TEquation}) can have at most one
solution.

Now, in order to show the existence of a solution of
(\ref{TEquation}), we must construct a sequence $\left(
\textbf{u}_{n}\right)$ in $X$ so that $\left(\textbf{Tu}_{n}
\right)$ converges to $\textbf{f}$ in $Y$.  In this regard, the
most general such result is given by Theorem \ref{Approx}.  As
such, within such a general context as considered here, it follows
that, if the mapping (\ref{PDEDefMap}) is
$\mathcal{C}^{k}$-smooth, we have
\begin{eqnarray}
X\supseteq \mathcal{ML}^{m+k}\left(\Omega\right)^{K}.\label{XMLmk}
\end{eqnarray}
It now follows by (\ref{T}) and (\ref{XMLmk}) that
\begin{eqnarray}
Y\supseteq \mathcal{ML}^{k}\left(\Omega\right)^{K}.\label{YMLmk}
\end{eqnarray}
This may be summarized in the following commutative diagram.
\begin{eqnarray}
\begin{array}{c}
\setlength{\unitlength}{1cm} \thicklines
%\begin{pspicture}(13,6)
\begin{picture}(13,6)

\put(2.4,5.0){$\mathcal{ML}^{m+k}\left(\Omega\right)^{K}$}
\put(4.6,5.1){\vector(1,0){5.9}}
\put(10.6,5.0){$\mathcal{ML}^{k}\left(\Omega\right)^{K}$}
\put(6.8,5.5){$\textbf{T}$} \put(3.2,4.8){\vector(0,-1){3.5}}
\put(3.4,0.8){\vector(1,0){7.3}} \put(3.0,0.7){$X$}
\put(10.9,0.7){$Y$} \put(2.6,3.0){$\subset$}
\put(11.2,3.0){$\subset$} \put(11.0,4.8){\vector(0,-1){3.5}}
\put(6.8,1.0){$\textbf{T}$}

\end{picture} \\
%\end{pspicture}
\end{array}\label{TTD}
\end{eqnarray}
Combining the diagram (\ref{TTD}) with (\ref{TTHat}) and
\begin{eqnarray}
\setlength{\unitlength}{1cm} \thicklines
\begin{array}{l}
\begin{picture}(14,6.5)

\put(1.9,5.0){$X$} \put(2.3,5.1){\vector(1,0){7.7}}
\put(10.1,5.0){$Y$} \put(6.2,5.3){$\textbf{T}$}
\put(2.2,4.7){\vector(1,-1){3.5}} \put(6.5,1.2){\vector(1,1){3.5}}
\put(5.8,0.7){$X_{\textbf{T}}$} \put(3.5,2.6){$q_{\textbf{T}}$}
\put(8.7,2.6){$\widehat{\textbf{T}}_{X}$}

\end{picture}
\end{array}\label{XTD}
\end{eqnarray}
we obtain an injective mapping
\begin{eqnarray}
\iota_{\textbf{T}}:\mathcal{ML}^{m+k}_{\textbf{T}}\left(\Omega\right)
\rightarrow X_{\textbf{T}}\label{IotaT}
\end{eqnarray}
so that the diagram
\begin{eqnarray}
\setlength{\unitlength}{1cm} \thicklines
\begin{array}{l}
\begin{picture}(14,6.5)

\put(2.4,5.0){$\mathcal{ML}^{m+k}_{\textbf{T}}\left(\Omega\right)$}
\put(4.6,5.1){\vector(1,0){5.9}}
\put(10.6,5.0){$\mathcal{ML}^{k}\left(\Omega\right)^{K}$}
\put(6.8,5.5){$\widehat{\textbf{T}}$}
\put(3.2,4.7){\vector(0,-1){3.5}} \put(3.7,0.8){\vector(1,0){7.1}}
\put(3.0,0.7){$X_{\textbf{T}}$} \put(10.9,0.7){$Y$}
\put(2.6,3.0){$\iota_{\textbf{T}}$} \put(11.2,3.0){$\subset$}
\put(11.0,4.8){\vector(0,-1){3.5}}
\put(6.8,1.0){$\widehat{\textbf{T}}_{X}$}

\end{picture}
\end{array}\label{XTDiagram}
\end{eqnarray}
commutes.  In particular, if the subspace convergence structure
induced on $\mathcal{ML}^{m+k}_{\textbf{T},k}\left(\Omega\right)$
from $Y$ is coarser than the order convergence structure, then the
mapping (\ref{IotaT}) is uniformly continuous. Furthermore, in
this case the mapping (\ref{IotaT}) extends to an injective
uniformly continuous mapping
\begin{eqnarray}
\iota_{\textbf{T}}^{\sharp}:\mathcal{ML}^{m+k}_{\textbf{T}}\left(\Omega\right)^{\sharp}
\rightarrow X_{\textbf{T}}^{\sharp}\label{IotaTSharp}
\end{eqnarray}
so that the extended diagram
\begin{eqnarray}
\setlength{\unitlength}{1cm} \thicklines
\begin{array}{l}
\begin{picture}(14,6.5)

\put(2.4,5.0){$\mathcal{ML}^{m+k}_{\textbf{T}}\left(\Omega\right)^{\sharp}$}
\put(4.6,5.1){\vector(1,0){5.9}}
\put(10.6,5.0){$\mathcal{ML}^{k}\left(\Omega\right)^{K}$}
\put(6.8,5.5){$\widehat{\textbf{T}}^{\sharp}$}
\put(3.2,4.7){\vector(0,-1){3.5}} \put(3.7,0.8){\vector(1,0){7.1}}
\put(3.0,0.7){$X_{\textbf{T}}^{\sharp}$} \put(10.9,0.7){$Y$}
\put(2.6,3.0){$\iota_{\textbf{T}}^{\sharp}$}
\put(11.2,3.0){$\subset$} \put(11.0,4.8){\vector(0,-1){3.5}}
\put(6.8,1.0){$\widehat{\textbf{T}}^{\sharp}_{X}$}

\end{picture}
\end{array}\label{ExtXTDiagram}
\end{eqnarray}
The existence of the injective mapping (\ref{IotaTSharp}) may be
interpreted as follows.  Any pullback type space of generalized
functions $X_{\textbf{T}}^{\sharp}$ which is constructed as above,
and subject to the condition of \textit{generality} of the
nonlinear partial differential operator $\textbf{T}$ must contain
the space $\mathcal{ML}^{m+k}_{\textbf{T}}\left(
\Omega\right)^{\sharp}$. As such, within the context of
\textit{general}, continuous systems of nonlinear PDEs, the
generalized functions in $\mathcal{ML}^{m+k}_{\textbf{T}}\left(
\Omega \right)^{\sharp}$ may be considered to be `more regular'
than those in any other space of generalized functions constructed
in this way.

\section{Conclusion}

In this paper we have shown that initial uniform convergence
structures are, in general, not preserved by the Wyler completion.
In particular, we discuss the completion of subspaces and products
of uniform convergence spaces in some detail.  Nevertheless, some
insight into the structure of the completion of an initial uniform
convergence space is obtained.

As an application of these results, we obtain the existence of
generalized solutions of arbitrary $\mathcal{C}^{k}$-smooth
systems of nonlinear PDEs.  In addition, a blanket regularity
result is obtained, in the sense that every generalized solution
may be assimilated with functions which are
$\mathcal{C}^{k}$-smooth everywhere except on a closed nowhere
dense set.  These results are shown to be maximal, with respect to
regularity, within the setting of the so called pullback spaces of
generalized functions used here and in \cite{Oberguggenberger},
\cite{vdWalt6} and \cite{vdWalt2}.


\begin{thebibliography}{999}

\bibitem{Anguelov and Rosinger 1} R. Anguelov and E. E. Rosinger, Solving large classes of nonlinear systems of
PDE's, \textit{Computers and Mathematics with Applications}
\textbf{53} (2007), 491-507.

\bibitem{Anguelov and van der Walt} R. Anguelov and J. H. van der Walt, Order convergence structure on $\mathcal{C}\left(X\right)$, \textit{Quaestiones
Mathematicae} \textbf{28} (2005), 425-457.

\bibitem{Beattie} R. Beattie and H. P. Butzmann, \textit{Convergence structures and applications to functional analysis}
Kluwer Academic Plublishers, Dordrecht, Boston, London (2002).

\bibitem{Dilworth} R. P. Dilworth, The normal completion of the lattice of continuous
functions, \textit{Transactions of the AMS} \textbf{68} (1950),
427-438.

\bibitem{Fric and Kent} R. Fric and D. C. Kent,
Completion of pseudotopological groups, \textit{Mathematische
Nachrichten} {bf 99} (1980), 99-103.

\bibitem{Gahler 2} W. G\"{a}hler W, \textit{Grundstrukturen der
anlysis II}, Birkh\"{a}user Verlag, Basel (1978).

\bibitem{GGK} S. G\"{a}hler, W. G\"{a}hler and G. Kneis,
Completion of pseudo-topological vector spaces
\textit{Mathematische Nachrichten} {bf 75} (1976), 185-206.

\bibitem{Kent} D. C. Kent and R. Ruiz de Eguino, On
products of Cauchy completions, \textit{Mathematische Nachrichten}
{\bf 155} (1992), 47-55.

\bibitem{RieszI} W. A. Luxemburg and A. C. Zaanen, \textit{Riesz Spaces I}
North-Holland, Amsterdam, London (1971).

\bibitem{Oberguggenberger} M. Oberguggenberger and E. E. Rosinger,
\textit{Solution of continuous nonlinear PDEs through order
completion}, North-Holland, Amsterdam, London, New York, Tokyo
(1994).

\bibitem{Reed} E. E. Reed, Completions of uniform convergence
spaces, \textit{Mathematische Annalen} \textbf{194} (1971),
83-108.

\bibitem{Rosinger 3} E. E. Rosinger, \textit{Nonlinear partial differential equations, an algebraic view of generalized
solutions}, North Holland Mathematics Studies, vol. 164 (1990).

\bibitem{vdWalt8} J. H. van der Walt, Order convergence on
Archimedean vector lattices with applications, MSc Thesis,
University of Pretoria, 2006.

\bibitem{vdWalt6} J. H. van der Walt, The uniform order convergence structure on $\mathcal{ML}\left(X\right)$, \textit{Quaestiones
Mathematicae} \textbf{31} (2008), 55-77.

\bibitem{vdWalt2} J. H. van der Walt, The order completion method
for systems of nonlinear PDEs:  Pseudotopological perspectives,
\textit{Acta Applicandae Mathematicae} \textbf{103} (2008), 1-17.

\bibitem{vdWalt3} J. H. van der Walt, The order completion method
for systems of nonlinear PDEs revisited, To Appear in \textit{Acta
Applicandae Mathematicae}.

\bibitem{vdWalt7} J. H. van der Walt, The order completion method
for systems of nonlinear PDEs:  Regularity of generalized
solutions, Technical Report UPWT 2008/??, University of Pretoria,
2008.

\bibitem{Wyler} O. Wyler, Ein komplettieringsfunktor f\"{u}r uniforme limesr\"{a}ume, \textit{Mathematische Nachrichten} \textbf{40} (1970), 1-12.

\end{thebibliography}
\end{document}